\documentclass[12pt]{amsart}
\usepackage{amssymb,amscd}
\usepackage{verbatim}

\let\Bbb\mathbb

\def\>{\relax\ifmmode\mskip.666667\thinmuskip\relax\else\kern.111111em\fi}
\def\<{\relax\ifmmode\mskip-.333333\thinmuskip\relax\else\kern-.0555556em\fi}
\def\vsk#1>{\vskip#1\baselineskip}
\def\vv#1>{\vadjust{\vsk#1>}\ignorespaces}
\def\vvn#1>{\vadjust{\nobreak\vsk#1>\nobreak}\ignorespaces}

  \let\ssize\scriptstyle
\let\sssize\scriptscriptstyle

\let\Medskip\medskip
\def\medskip{\par\Medskip}
\let\Bigskip\bigskip
\def\bigskip{\par\Bigskip}

\let\Maketitle\maketitle
\def\maketitle{\Maketitle\thispagestyle{empty}\let\maketitle\empty}

\newtheorem{thm}{Theorem}[section]
\newtheorem{cor}[thm]{Corollary}
\newtheorem{lem}[thm]{Lemma}

\numberwithin{equation}{section}

\theoremstyle{definition}
\newtheorem*{rem}{Remark}
\newtheorem*{example}{Example}

\let\mc\mathcal
\let\nc\newcommand

\let\ka\kappa
\let\la\lambda
\let\La\Lambda

\let\phi\varphi
\let\si\sigma

\let\om\omega

\let\der\partial

\let\leq\leqslant

\let\on\operatorname
\let\bi\bibitem
\let\bs\boldsymbol

\def\C{{\mathbb C}}

\def\B{{\mc B}}

\def\F{{\mc F}}

\def\+#1{^{\{#1\}}}

\def\End{\on{End}}

\def\gln{\mathfrak{gl}_N}
\def\sln{\mathfrak{sl}_N}

\def\beq{\begin{equation}}
\def\eeq{\end{equation}}
\def\be{\begin{equation*}}
\def\ee{\end{equation*}}

\nc{\bea}{\begin{eqnarray*}}
\nc{\eea}{\end{eqnarray*}}
\nc{\bean}{\begin{eqnarray}}
\nc{\eean}{\end{eqnarray}}
\nc{\Ref}[1]{{\rm(\ref{#1})}}

\let\ga\gamma

\nc{\Il}{{\mc I_{\bs\la}}}
\nc{\bla}{{\bs\la}}
\nc{\Fla}{\F_\bla}
\nc{\tfl}{{T^*\Fla}}
\nc{\GL}{{GL_n(\C)}}
\nc{\GLC}{{GL_n(\C)\times\C^*}}

\let\sd s 

\def\ddk_#1{\kk_{#1}\<\>\frac\der{\der\<\>\kk_{#1}}}

\def\bul{\mathbin{\raise.2ex\hbox{$\sssize\bullet$}}}
\def\intt{\mathchoice
{\mathop{\raise.2ex\rlap{$\,\,\ssize\backslash$}{\intop}}\nolimits}
{\mathop{\raise.3ex\rlap{$\,\sssize\backslash$}{\intop}}\nolimits}
{\mathop{\raise.1ex\rlap{$\sssize\>\backslash$}{\intop}}\nolimits}
{\mathop{\rlap{$\sssize\<\>\backslash$}{\intop}}\nolimits}}

\let\kk q 
\let\cc c

\let\Ko K

\def\GZ/{Gelfand-Zetlin}
\def\KZ/{{\slshape KZ\/}}
\def\qKZ/{{\slshape qKZ\/}}
\def\XXX/{{\slshape XXX\/}}

\nc{\slnl}{{\sln (\lambda)}}
\nc{\PCN}{{   (\C[x])^N   }}
\nc{\di}{\on{Diag}}
\nc{\dio}{\on{Diag}_0}
\nc{\Mm}{{\mc M}}
\nc{\Nn}{{\mc N}}

\nc{\A}{{\mc C}}

\nc{\PCr}{{  P  (\C[x])^n   }}

\nc{\Pk}{{(\bs{P}^1)^k}}

\nc{\N}{{\Bbb N}}

\nc{\Ll}{{\mc L}}

\nc{\ord}{{\on{ord}\,}}

\nc{\Sing}{{\on{Sing}\,}}
\nc{\sing}{{\on{Sing}\,}}

\nc{\Hess}{{\on{Hess}}}

\nc{\R}{{\Bbb R}}

\let\on\operatorname
\nc{\Kk}{{\bs K}}
\nc{\Ap}{{A_\Phi(z)}}
\nc{\ap}{{A_\Phi(z)}}

\nc{\sv}{{\sing V}}
\nc{\cd}{{\C^n-\Delta}}
\nc{\UT}{{U^0}}   
\nc{\Spect}{\on{Spec}\nolimits}

\begin{document}

\hrule width0pt
\vsk->

\title[Characteristic variety  of Gauss-Manin differential equations ]
{Characteristic variety of the Gauss-Manin differential equations
of a generic parallelly translated  arrangement }

\author
[Alexander Varchenko ]
{ Alexander Varchenko$\>^{\star}$}

\maketitle

\begin{center}
{\it Department of Mathematics, University of North Carolina
at Chapel Hill\\ Chapel Hill, NC 27599-3250, USA\/}
\end{center}

{\let\thefootnote\relax
\footnotetext{\vsk-.8>\noindent
$^\star$\,{\it E-mail}: anv@email.unc.edu,  supported in part by NSF grant DMS--1101508}}

\medskip

\begin{abstract}
We consider a weighted family of $n$ generic parallelly translated hyperplanes in $\C^k$  and describe  the characteristic variety of the Gauss-Manin
differential equations for associated hypergeometric integrals.
The characteristic variety is given as the zero set of Laurent polynomials, whose coefficients are determined by weights and
the Pl\"ucker coordinates of the associated point in the Grassmannian Gr$(k,n)$.
The Laurent polynomials are in involution.
\end{abstract}

{\small \tableofcontents  }

\setcounter{footnote}{0}
\renewcommand{\thefootnote}{\arabic{footnote}}

\section{Introduction}
There are three places, where a flat connection depending on a parameter appears:

\noindent
$\bullet$\  KZ equations,
$\kappa \frac{\der I}{\der z_i}(z) = K_i(z) I(z)$,  $z=(z_1,\dots,z_n)$, $i=1,\dots,n$.
Here $\kappa$ is a parameter, $I(z)$  a $V$-valued function, where  $V$ is a vector space from representation theory,
$K_i(z):V\to V$ are linear operators, depending on $z$. The connection is flat for all $\kappa$, see for example \cite{EFK, V1}.

\noindent
$\bullet$\ Quantum differential equations,
$\kappa \frac{\der I}{\der z_i}(z) = p_i *_z I(z)$,  $z=(z_1,\dots,z_n)$, $i=1,\dots,n.$
Here $p_1,\dots,p_n$ are generators of some commutative algebra $H$ with quantum multiplication $*_z$ depending on $z$.
The connection is flat for all $\kappa$.
These equations are part of the Frobenius structure on the quantum cohomology of a variety, see \cite{D, M}.

\noindent
$\bullet$\ Differential equations for hypergeometric integrals associated with a family of weighted arrangements with parallelly
 translated hyperplanes,
$\kappa \frac{\der I}{\der z_i}(z) = K_i(z) I(z)$,  $ z=(z_1,\dots,z_n)$, $ i=1,\dots,n$. The connection is flat for all $\kappa$,
see for example \cite{V5, OT}.

\smallskip
If $\kappa \frac{\der I}{\der z_i}(z) = K_i(z) I(z)$, $i=1,\dots,n$, is a system of $V$-valued differential equations of one of these types, then
its characteristic variety is 
\be
\Spect = \{(z,p)\in T^*\C^n\ |\ \exists
v\in V\ \text{with}\ K_j(z)v = p_jv,\ j\in J\}.
\ee
It is known that the characteristic varieties of the first two types of differential equation are interesting. For example,
the characteristic variety of the quantum differential equation of the flag variety is the zero set of the Hamiltonians of the classical Toda lattice,
according to \cite{G, GK}, and the characteristic variety of the $\gln$  KZ equations with values in the tensor power of the vector representation  is the zero
set of the Hamiltonians of the classical Calogero-Moser system, according to \cite{MTV}.

In this paper we describe the characteristic variety of the Gauss-Manin differential equations for hypergeometric integrals associated with
a weighted family of $n$ generic parallelly translated hyperplanes in $\C^k$.
The characteristic variety is given as the zero set of Laurent polynomials, whose coefficients are determined by weights and
the Pl\"ucker coordinates of the associated point in the Grassmannian Gr$(k,n)$.
The Laurent polynomials are in involution.

It is known that the KZ differential equations can be identified with Gauss-Manin differential equations of certain weighted families of
parallelly translated hyperplanes, see \cite{SV}, as well as some  quantum differential equations
can be identified with Gauss-Manin differential equations of certain weighted families of
parallelly translated hyperplanes, see \cite{TV}.
Therefore, the results in this paper on the characteristic variety of the Gauss-Manin
differential equations associated with a family   of generic parallelly translated hyperplanes
can be considered as a first step to studying characteristic varieties of more general KZ and quantum
 differential equations, which admit integral hypergeometric representations.

The Laurent polynomials, defining our characteristic variety,  are regular functions of the Pl\"ucker coordinates
of the associated point in  $\on{Gr}(k,n)$. Therefore they can be used to study the characteristic
varieties of more general Gauss-Manin differential equations for multidimensional hypergeometric integrals.

Our description of the characteristic variety is based on the fact, proved in \cite{V3},
that the characteristic variety of the Gauss-Manin differential equations is generated by the master function
of the corresponding hypergeometric integrals, that is, the characteristic variety coincides with the Lagrangian variety
of the master function. That fact is a generalization of Theorem 5.5 in \cite{MTV}, proved with the help of the Bethe ansatz, that
the local algebra of a critical point of the master function associated with a $\gln$ KZ equation can be identified with
a suitable local Bethe algebra of the corresponding $\gln$ module.

\smallskip
In Section \ref{sec 2}, we consider the algebra of functions on the critical set of the master function and
describe it by generators and relations.

In Section \ref{ sec Lagrangian var function}, we show that these relations give
us equations defining the Lagrangian variety of the master function. We show that the corresponding functions are in involution.
We define coordinate systems $(z_I,p_{\bar I})$ on the Lagrange variety and for each of them a function $\Phi(z_I,p_{\bar I})$ 
also generating the Lagrangian variety.
We describe the Hessian of the master function lifted to the Lagrangian variety and relate it to the Jacobian of the
projection of the  Langrangian variety
to the base of the family.

In Section \ref{sec last}, we remind the identification from \cite{V3} of the Lagrangian variety of the master function
and the characteristic variety of the Gauss-Manin differential equations.


\section{Algebra of functions on the critical set}
\label{sec 2}

\subsection{An arrangement in  $\C^n\times\C^k$}
\label{An arrangement in ck}
Let $n>k$ be positive integers. Denote $J=\{1,\dots,n\}$.
Consider $\C^k$ with coordinates $t_1,\dots,t_k$,\
$\C^n$ with coordinates $z_1,\dots,z_n$. Fix $n$
 linear functions on $\C^k$,
$g_j=\sum_{m=1}^k b^m_jt_m,$\ $ j\in J,$
$b_j^m\in \C$. For ${i_1,\dots,i_k}\subset J$,
denote
$d_{i_1,\dots,i_k} = \text{det}_{\ell,m=1}^k (b^m_{i_\ell})$.
We assume that all the numbers $d_{i_1,\dots,i_k}$ are nonzero if ${i_1,\dots,i_k}$ are distinct. In other words,
 we assume that the collection of functions $g_j, j\in J$, is generic.
We define $n$ linear functions on $\C^n\times\C^k$,
$f_j = z_j+g_j,$\ $ j\in J.$
 We define
 the arrangement of hyperplanes $\tilde \A = \{ \tilde H_j\ | \  j\in J \}$ in $\C^n\times \C^k$, where $\tilde H_j$ is the zero set of $f_j$.
Denote by $ U(\tilde\A) = \C^n\times\C^k - \cup_{j\in J} \tilde H_j$ the complement.

For every $z=(z_1,\dots,z_n)\in \C^n$, the arrangement $\tilde \A$
induces an arrangement $\A(z)$ in the fiber over $z$ of the projection $\pi : \C^n\times\C^k\to\C^n$. We
identify every fiber with $\C^k$. Then $\A(z)$ consists of
hyperplanes  $H_j(z), j\in J$, defined in $\C^k$ by the equations
$f_j=0$.
Denote by $ U(\A(z)) = \C^k - \cup_{j\in J} H_j(z)$  the complement.

The arrangement $\A(z)$ is with normal crossings if and only if $z \in \C^n-\Delta$,
\bean
\Delta = \cup_{\{i_1<\dots<i_{k+1}\}\subset J}H_{i_1,\dots,i_{k+1}},
\eean
where $H_{i_1,\dots,i_{k+1}}$ is the hyperplane in $\C^n$ defined by the equation
$f_{i_1,\dots,i_{k+1}}(z)=0$,
\bean
\label{ i_1,...,i_k}
f_{i_1,\dots,i_{k+1}}(z)= \sum_{m=1}^{k+1} (-1)^{m-1}  d_{i_1,\dots,\widehat{i_m},\dots,i_{k+1}} z_{i_m}.
\eean
We have the following identify
\bean
\label{id 1}
\sum_{m=1}^{k+1} (-1)^{m-1} d_{i_1,\dots,\widehat{i_m},\dots,i_{k+1}} (z_{i_m} - f_{i_m}(z,t))  = 0.
\eean

\begin{lem}
\label{dim S}
Consider the $\C$-span $S$ of the linear functions $f_{i_1,\dots,i_{k+1}}$, where $\{i_1,\dots,i_{k+1}\}$ runs through all
$k+1$-element subsets of $J$. Then $\dim S = n-k$.
\end{lem}

\begin{proof} The dimension of $S$ equals the codimension in $\C^n$  of $X_1=\{z\in\C^n\  |\ f_I(z)= 0\ \text{for\,all}\ I\}$.
The subspace $X_1$ is the image of the subspace $X_2=\{(z,t)\in \C^n \times \C^k \ | \ f_j(z,t)=0 \ \text{for\,all}\ j\in J\}$ under the projection
$\pi : \C^n\times\C^k\to \C^n$. Clearly the subspace $X_2$ is $k$-dimensional and the projection
$\pi|_{X_2} : X_2\to X_1$ is an isomorphism. Hence $\dim X_1=k$ and $\dim S=n-k$.
\end{proof}

\subsection{Pl\"ucker coordinates}
The matrix $(b^m_j)$ is an $n\times m$-matrix of rank $k$. The matrix defines a point
in the Grassmannian $\on{Gr}(k,n)$ of $k$-planes in $\C^n$. The numbers $d_{i_1,\dots,i_k}$ are Pl\"ucker coordinates of this point.
Most of objects in this paper is determined in terms of these  Pl\"ucker coordinates.
We will use the following Pl\"ucker relation.

\begin{lem}
\label{plucker}

For arbitrary sequences $j_1,\dots,j_{k+1}$ and $i_1,\dots, i_{k-1}$  in $J$, we have
\bean
\label{Plu}
\sum_{m=1}^{k+1} (-1)^{m-1} d_{j_1,\dots,\widehat{j_m},\dots,j_{k+1}} d_{j_m,i_{1}\dots,i_{k-1}} =0.
\eean
\end{lem}

See this statement, for example, in \cite{KL}.


\subsection{Algebra $A_\Phi(z)$}
\label{sec master k}
Assume that nonzero weights $(a_j)_{j\in J}\subset \C^\times$ are given.
Denote $|a|=\sum_{j\in J}a_j$. Assume that $|a|\ne 0$.

Each arrangement $\A(z)$  is weighted.
The master function of the weighted arrangement $\A(z)$ in $\C^k$ is the function
\bean
\label{def mast k}
\Phi(z,t) = \sum_{j\in J}\,a_j \log f_j(z,t).
\eean
The critical point equations are
\bean
\label{CR}
{\der\Phi}/{\der t_i} = \sum_{j\in J} b^i_j {a_j}/{f_j} = 0, \qquad i=1,\dots,k.
\eean
We  have
  \bean
  \label{der z}
{  \der \Phi}/{\der z_j} = a_j/f_j, \qquad i\in J.
\eean
Denote by $\mc I(z)\subset \mc O(U(\A(z)))$ the ideal generated by the functions $\der \Phi/\der t_j$, $j\in J$.
The algebra of functions on the critical set is
\bean
A_\Phi(z) = \mc O(U(\A(z)))/\mc I(z).
\eean
For a function $g\in \mc O(U(\A(z)))$, denote by $[g]$ its projection to $A_\Phi(z)$. Denote
\bea
p_j  =[a_j/f_j], \qquad j\in J.
\eea

We introduce the following polynomials in
$z_1,\dots,z_n,p_1,\dots,p_n$.
For every subset $I=\{i_1,\dots,i_{k-1}\}$ of distinct elements in $J$, we set
\bean
\label{rel 1st type}
F_I(p_1,\dots,p_n) = \sum_{j\in J} d_{j,i_1,\dots,i_{k-1}} p_j .
\eean
For every subset $I=\{i_1,\dots,i_{k+1}\}$ of distinct elements in $J$, we set
\bean
\label{rel 2nd type}
&&
F_I(z_1,\dots,z_n,p_1,\dots,p_n) =
\\
&&
\notag
\phantom{aaaaa}
= p_{i_1}\dots p_{i_{k+1}}\,f_{i_1, i_2,\dots,i_{k+1}}(z)
+
\sum_{m=1}^{k+1} (-1)^{m} a_{i_m} d_{i_1,\dots,\widehat{i_m},\dots,i_{k+1}} p_{i_1}\dots \widehat{p_{i_m}}\dots p_{i_{k+1}} .
\eean

The following lemma collects  properties of the elements $p_1,\dots,p_n$.

\begin{lem}
\label{p properties}
Let $z\in\C^n-\Delta$.

\begin{enumerate}
\item[(i)]
The elements $p_j, j\in J$, generate the algebra $A_\Phi(z)$.

\item[(ii)]
For every subset $I=\{i_1,\dots,i_{k-1}\}$ of distinct elements in $J$, we have
\bean
\label{rel 1st type}
F_I(p_1,\dots,p_n)=0.
\eean
Relation \Ref{rel 1st type} will be called the $I$-{\it relation of first kind.}

\item[(iii)]
For every subset $I=\{i_1,\dots,i_{k+1}\}$ of distinct elements in $J$, we have
\bean
\label{rel 2nd type}
F_I(z_1,\dots,z_n,p_1,\dots,p_n)=0.
\eean
Relation \Ref{rel 2nd type} will be called the $I$-{\it relation of second kind.}

\item[(iv)]

In $A_\Phi(z)$, we have
\bean
\label{1}
1=\frac 1{|a|}\sum_{j\in J} z_jp_j.
\eean

\item[(v)]
We  have $\dim A_\Phi(z) = {n-1\choose k}$, and for any $j_1\in J$, the set of monomials $p_{i_1}\dots p_{i_k}$, with  $i_1<\dots<i_k$ and $j_1\notin\{i_1,\dots,i_k\}$,
is a $\C$-basis of $A_\Phi(z)$.

\end{enumerate}
\end{lem}

Part (i) is Lemma 2.5 in \cite{V3}. Parts (ii), (iii), (iv) are Lemmas 6.7, 6.8, 2.5 in \cite{V4}, respectively.
The first statement of part (v) is \cite[Lemma 4.2]{V3} that follows from   \cite[Lemma 6.5]{V4}.
The second statement of part (v) is  Theorem 6.11 in \cite{V4}.

\smallskip
Note that the polynomials $F_I$ in  \Ref{rel 1st type} and \Ref{rel 2nd type} are homogeneous if we put
\bean
\label{homog}
\deg p_j=1,
\qquad \deg z_j=-1 \qquad \text{for\, all}\ j.
\eean

\subsection{Relations of second kind}
For $j\in J$, denote
\bean
\label{G}
G_j(z_j,p_j) = z_j - a_j/p_j.
\eean
Then the projection to $A_\Phi(z)$ of the left hand side of equation \Ref{id 1} can be written as
\bean
\label{id 2}
G_I(z,p) &=& \sum_{m=1}^{k+1} (-1)^{m-1}  d_{i_1,\dots,\widehat{i_m},\dots,i_{k+1}} G_{i_m}(z_{i_m},p_{i_m})
\phantom{aaaaaaaaaaaaa}
\\
\notag
&=&
\sum_{m=1}^{k+1} (-1)^{m-1}  d_{i_1,\dots,\widehat{i_m},\dots,i_{k+1}} \Big(z_{i_m} - \frac{a_{i_m}}{p_{i_m}}\Big),
\eean
where $I=\{i_1,\dots,i_{k+1}\}$. Hence in $A_\Phi(z)$ we have
\bean
\label{G=0}
G_I(z,p)  = 0.
\eean
Notice that $F_I(z,p) = p_{i_1}\dots p_{i_{k+1}} G_I(z,p)$ and
the functions $p_j $ are nonzero at every point of the critical set of the master function.

\subsection{New presentation for $A_\Phi(z)$}

Fix $z\in \C^n-\Delta$. Consider $(\C^\times)^n$ with coordinates $p_1,\dots,p_n$.
Consider the polynomials $F_I(p)$ in \Ref{rel 1st type} and polynomials $F_I(z,p)$ in \Ref{rel 2nd type} as elements of $\mc O((\C^\times)^n)$.
 Let $\tilde{\mc I}(z)\subset \mc O((\C^\times)^n)$ be the ideal generated by all  $F_I$ with $|I|=k-1,\ k+1$.

Notice that all polynomials $F_I(p)$, $|I|=k-1$, in \Ref{rel 1st type} and all functions $G_I(z,p), |I|=k+1$, in \Ref{id 2}
also generate $\tilde{\mc I}(z)$.

 Let $\tilde A(z) = \mc O((\C^\times)^n)/\tilde{\mc I}(z)$ be the quotient algebra.

\begin{thm}
\label{thm main}
The natural homomorphism   $\tilde A(z)\to A_\Phi(z)$, $p_j\mapsto [a_j/f_j]$,  is an isomorphism.
\end{thm}

\begin{example} If $k=1$ and $f_j=t_1+z_j$, then the ideal $\mc I(z)$ is generated by the function
$\sum_{j\in J} a_j/(t_1+z_j)$, while the ideal $\tilde{\mc I}(z)$ is generated by the functions
\bea
p_1+\dots+p_n, \qquad (z_i-z_j)p_ip_j - a_ip_j + a_jp_i, \quad 1\leq i<j\leq n,
\eea
or by the functions
\bea
p_1+\dots+p_n, \qquad (z_i-a_i/p_i)- (z_j-a_j/p_j), \quad 1\leq i<j\leq n.
\eea

\end{example}

\subsection{Proof of Theorem \ref{thm main}}

\begin{lem}
\label{lem 1 f}
Let $I=\{i_1,\dots,i_k\}$ be a subset of distinct elements. Then in $\tilde A(z)$, we have
\bean
\label{1 I}
\sum_{j\in J} z_jp_j = \frac1{d_{i_1,\dots,i_k}} \sum_{j\in J-I}   f_{j,i_1,\dots,i_k}(z)\,p_j.
\eean

\end{lem}
\begin{proof}
The statement easily follows from \Ref{rel 1st type}, that is, from relations of first kind. For example, if $k=2$ and $I=\{1,2\}$, then the two relations of first kind
$p_1=\frac 1{d_{2,1}}\sum_{j>2} d_{j,2}p_j$ and
$p_2=\frac 1{d_{1,2}}\sum_{j>2} d_{j,1}p_j$
transform $\sum_{j\in J} z_jp_j$ to $\frac1{d_{1,2}} \sum_{j>2}   f_{1,2,j}(z) p_j$.
\end{proof}

\begin{lem}
\label{lem 1 in tilde}
In $\tilde A(z)$, we have
$1= \frac 1{|a|}\sum_{j\in J} z_jp_j$.
\end{lem}

\begin{proof}
We have
\bea
&&
p_1\dots p_k \sum_{j\in J} z_jp_j = p_1\dots p_k \frac1{d_{1,\dots,k}} \sum_{j>k}   f_{j,1,\dots,k}(z) p_j
\\
&&
\phantom{aaa}
=   \sum_{j>k} \big[a_jp_1\dots p_k + \sum_{m=1}^{k} (-1)^{m} a_{m} \frac{ d_{j,1,\dots,\widehat{m},\dots,k}}{d_{1,\dots,k}}p_j p_{1}\dots \widehat{p_{m}}\dots p_{k} \big]
= |a| \,p_1\dots p_k ,
\eea
where the first equality follows from Lemma \ref{lem 1 f}, the second equality follows from the relations of second kind, the third equality follows from the relations of first kind.
Denote by $C(z)\subset (\C^\times)^n$ the zero set of the ideal $\tilde{\mc I}(z)$. Then the function $p_1\dots p_k$ is nonvanishing on $C(z)$. The previous calculation
shows that the multiplication of the invertible function $p_1\dots p_k$ by $\frac 1{|a|}\sum_{j\in J} z_jp_j$ does not change the invertible function. This gives the lemma.
\end{proof}

\begin{lem}
\label{lem C comb}
Let $s\leq k$ be a natural number and $M=\prod_{j\in J} p_j^{s_j}$, \ $\sum_{j\in J}s_j = s$, a monomial of degree $s$.
Let $J_{k-s+1}=\{j_1,\dots,j_{k-s+1}\}$ be any subset in $J$ with distinct elements. Then by using  the relations of
first kind only, the monomial $M$ can be represented as a $\C$-linear combination of monomials $p_{i_1}\dots p_{i_s}$ with $1\leq i_1 <\dots<i_s\leq n$
and $\{ i_1,\dots,i_s\}\cap J_{k-s+1}=\emptyset$.
\qed
\end{lem}

C.f. the proof of Lemma 6.9 in \cite{V4}.

\begin{lem}
\label{lem polyn comb 1}
Let $s\leq k$ be a natural number and $M=\prod_{j\in J} p_j^{s_j}$
a monomial of degree $s$.
Fix and element $j_1\in J$.  Then by using  the relations of
first kind and the relation $1= \frac 1{|a|}\sum_{j\in J} z_jp_j$ only, the monomial $M$ can be represented as a linear combination of monomials $p_{i_1}\dots p_{i_k}$ with $1\leq i_1 <\dots<i_k\leq n$
and $j_1\notin \{ i_1,\dots,i_s\}$, where the coefficients of the linear combination are homogeneous polynomials in $z$ of degree $s-k$.
\qed
\end{lem}

Recall the $\deg z_j=-1$ for all $j\in J$.

\begin{lem}
\label{lem polyn comb 2}
Let $s > k$ be a natural number and $M=\prod_{j\in J} p_j^{s_j}$
a monomial of degree $s$.
 Then by using  the relations of
first kind and  second kinds, the monomial $M$ can be represented as a linear combination of monomials $p_{i_1}\dots p_{i_k}$ of degree $k$,
where the coefficients of the linear combination are rational functions  in $z$, regular on $\C^n-\Delta$ and homogeneous of degree  $s-k$.
\qed
\end{lem}

Let us finish the proof of Theorem \ref{thm main}. Let  $P(p_1,\dots,p_n)$ be a polynomial. Fix $j_1\in J$.
By using the relations of first and second kinds only, the polynomial can be represented as a linear combination $\tilde P$
of monomials $p_{i_1}\dots p_{i_k}$ with $1\leq i_1 <\dots<i_k\leq n$
and $j_1\notin \{ i_1,\dots,i_s\}$, see Lemmas \ref{lem C comb} - \ref{lem polyn comb 2}.
Assume that $P(p_1,\dots,p_n)$ projects to zero in $A_\Phi(z)$, then all coefficients of that linear combination $\tilde P$ must be
zero, see part (v) of Lemma \ref{p properties}. This means that $P$ lies in the ideal $\tilde{\mc I}(z)$. Theorem
\ref{thm main} is proved.

\section{Lagrangian variety of the master function}
\label{ sec Lagrangian var function}

\subsection{Critical set}

Recall the projection $\pi :\C^n\times\C^k\to\C^n$. For any $z\in\C^n-\Delta$, the arrangement $\A(z)$ in $\pi^{-1}(z)$ has normal crossings.
Recall   the complement $U(\tilde \A)\subset \C^n\times\C^k$ to the arrangement $\tilde A$ in  $\C^n\times\C^k$.  Denote
\bean
\label{U}
U^0= U(\tilde \A) \cap \pi^{-1}(\C^n-\Delta) \subset \C^n\times\C^k.
\eean
Consider the master function $\Phi(z,t)$, defined in \Ref{def mast k}, as a function on $U^0$.
Denote by $C_\Phi$ the critical set of $\Phi$ with respect to variables $t$,
\bean
\label{crit big}
C_\Phi=\{(z,t)\in U^0\ |\ \der \Phi/\der t_i(z,t)=0,\ {} i=1,\dots,k\}.
\eean

\begin{lem}
\label{crit smooth}
The set $C_\Phi$ is a smooth $n$-dimensional subvariety of $\UT$.
\end{lem}

\begin{proof}
For any subset $I=\{1\leq i_1<\dots<i_k\leq n\}\subset J$, the  $k\times k$-determinant
\bea
\on{det}_{l,m=1}^k \Big(\frac{\der^2\Phi}{\der t_l\der z_{j_m}}\Big)\,= \,- d_{i_1,\dots,i_k} \prod_{m=1}^k \frac{a_{j_m}}{f^2_{j_m}(z,t)}
\eea
is nonzero on $\UT$.
\end{proof}

Denote by $\mc I\subset \mc O(U^0)$ the ideal generated by the functions $\der \Phi/\der t_j$, $j\in J$.
The algebra of functions on $C_\Phi$ is the quotient algebra
\bean
A_\Phi  =  \mc O(U^0)/\mc I.
\eean
Consider $(\C^n-\Delta)\times (\C^\times)^n$ with coordinates $z_1,\dots,z_n, p_1,\dots,p_n$.
Consider the polynomials $F_I(p)$ in \Ref{rel 1st type} and polynomials $F_I(z,p)$ in \Ref{rel 2nd type} as elements of $\mc O((\C^n-\Delta)\times (\C^\times)^n)$.
 Let $\tilde{\mc I}\subset \mc O((\C^n-\Delta)\times (\C^\times)^n)$ be the ideal generated by all  $F_I$ with $|I|=k-1,\ k+1$.
Notice that all polynomials $F_I(p)$, $|I|=k-1$, in \Ref{rel 1st type} and all functions $G_I(z,p), |I|=k+1$, in \Ref{id 2}
also generate $\tilde{\mc I}(z)$.  Let
 \bean
 \label{tilde A}
 \tilde A = \mc O((\C^n-\Delta)\times (\C^\times)^n)/\tilde{\mc I}
 \eean
  be the quotient algebra.

\begin{thm}
\label{thm main 2}
The natural homomorphism   $\tilde A \to A_\Phi $, $p_j\mapsto [a_j/f_j]$,  is an isomorphism.

\end{thm}

The proof is the same as the proof of Theorem \ref{thm main}.

\subsection{Lagrangian variety}
Consider the cotangent bundle $T^*(\C^n-\Delta)$ with dual coordinates $z_1,\dots,$ $z_n$,
$p_1,\dots,p_n$ with respect to  the standard symplectic form
$\om = \sum_{j=1}^n dp_j\wedge dz_j$. Consider the open subset $(\C^n-\Delta)\times (\C^\times)^n\subset T^*(\C^n-\Delta)$
of all points with nonzero coordinates $p_1,\dots,p_n$.
Consider  the map
\bea
\phi  :  C_\Phi\ \to\ (\C^n-\Delta)\times (\C^\times)^n,
\ (z,t) \mapsto  \Big(z_1,\dots,z_n, p_1=\frac{\der\Phi}{\der z_1}(z,t),\dots,p_n=\frac{\der\Phi}{\der z_n}(z,t)\Big).
\eea
Denote by $\La$ the image $\phi(C_\Phi)$ of the critical set. The set $\La$ is invariant with respect to the action of $\C^\times$ which multiplies
all coordinates $p_j$ and divides all coordinates $z_j$ by the same number.
Denote by  $\hat{\mc I}\subset \mc O((\C^n-\Delta)\times (\C^\times)^n)$ the ideal of functions   that equal zero on $\La$.

\begin{thm}
\label{thm main 3}
The ideal $\tilde {\mc I}\subset \mc O((\C^n-\Delta)\times (\C^\times)^n)$ coincides with the ideal $\hat{\mc I}$.
The subset $\La\subset (\C^n-\Delta)\times (\C^\times)^n$
is a smooth Lagrangian subvariety.
\end{thm}

\begin{proof}
It is clear that $\tilde{\mc I}\subset \hat{\mc I}$. The proof of the inclusion $\hat{\mc I}\subset \tilde{\mc I}$ is  basically
the same as the proof of
Theorem \ref{thm main}. This gives the first statement of the theorem.

It is clear that $\dim \La=n$. To prove that $\La$ is smooth, it is enough to show that at any point of $\La$, the span of the differentials
of the functions $F_I(p), |I|=k-1,$ and $G_I(z,p), |I|=k+1$ is at least $n$-dimensional. By Lemma \ref{dim S}, the span of the $z$-parts of the
differentials of the functions $G_I(z,p)$, $I=|I|=k+1$, is $n-k$-dimensional. It is easy to see that the span of the
differentials of the functions $F_I(p)$, $I=|I|=k+1$, is at least $k$-dimensional, c.f. the example in the proof of
Lemma \ref{lem 1 f}. Hence $\La$ is smooth.

By the definition of  $\phi$, the set $\La$ is isotropic. Hence $\La$ is Lagrangian.
\end{proof}

Let $I=\{i_1,\dots,i_k\}\subset J$ be a $k$-element subset and $\bar I$ its complement. Then
the functions $z_I=\{z_i\,|\, i\in I\}$, $p_{\bar I}=\{p_j\,|\, j\in \bar I\}$, form a system of coordinates on $\La$. Indeed, we have
\bean
\label{coord}
p_{i_m} &=& - \frac 1{d_{i_m, i_1,\dots,\widehat{i_m},\dots,i_k}}\sum_{j\in\bar I} d_{j, i_1,\dots,\widehat{i_m},\dots,i_k} p_j, \qquad m=1,\dots,k,
\\
\notag
z_j
& =&
 \frac{a_j}{p_j} + \frac 1{d_{i_1,\dots,i_k}} \sum_{m=1}^k(-1)^{k-m}
 d_{j, i_1,\dots,\widehat{i_m},\dots,i_{k}} \Big(z_{i_m} - \frac{a_{i_m}}{p_{i_m}}\Big), \qquad j\in\bar I,
\eean
where in the second line the functions $p_{i_m}$ must be expressed in terms of the functions $p_j, j\in\bar I$, by using the first line.

We order the functions of the coordinate system $z_I,p_{\bar I}$ according to the increase of the low index. For example, if $k=3, n=6$,
$I=\{1,3,6\}$, then the order is $z_1,p_2,z_3,p_4,p_5,z_6$.

\begin{lem}
\label{lem trans fns}
Let $I=\{i_1,\dots,i_k\}$ and $I'=\{i_1',\dots,i_k'\}$ be two $k$-element subsets of $J$.
Consider the corresponding ordered coordinate systems
  $z_I$, $p_{\bar I}$ and $z_{I'}$, $p_{\bar {I'}}$. Express the coordinates of the second system in terms of
coordinates of the first system and denote by $\on{Jac}_{I,\bar{I'}}(z_I,p_{\bar I})$ the Jacobian of this change. Then
\bea
\on{Jac}_{I,\bar{I'}}(z_I,p_{\bar I}) = (d_{i_1',\dots,i_k'}/d_{i_1,\dots,i_k})^2.
\eea

\end{lem}
\begin{proof}
It is enough to check this formula for the case $I=\{1,3,\dots,k+1\}$ and $I'=\{2,3,\dots,k+1\}$. Then
\bea
p_1 = - \frac{d_{2,3,\dots,k+1}} {d_{1,3,\dots,k+1}} p_2 + \dots ,
\qquad
z_{2}
 =
 \frac{a_{2}}{p_{2}} + \frac{d_{2,3,\dots,k+1}} {d_{1,3,\dots,k+1}}z_1 + \dots ,
\eea
where the first dots denote the terms which do not
depend on $z_1,p_2$ and the second dots denote the terms which do not depend on $z_1$.
According to these formulas the $2\times 2$ Jacobian of the dependence of $p_1,z_2$ on $z_1,p_2$ equals
$(d_{2,3,\dots,k+1}/d_{1,3,\dots,k+1})^2$ and hence $\on{Jac}_{I,\bar{I'}}(z_I,p_{\bar I}) =
(d_{2,3,\dots,k+1}/d_{1,3,\dots,k+1})^2$.
\end{proof}

\subsection{Generating functions}

Consider the function
\bean
\label{big fn}
\Psi = \sum_{j\in J} a_j\ln p_j - \sum_{i\in I} z_{i}p_{i}
\eean
of $n+k$ variables $z_{j}, j\in I$, \, $p_j, j\in J$. Express in $\Psi$ the variables $p_i, i\in I$, according to \Ref{coord}.
Denote by $\Psi(z_I,p_{\bar I})$ the resulting function of variables $z_I$, $p_{\bar I}$.

\begin{thm}
\label{thm Gener}
The function $\Psi(z_I,p_{\bar I})$ is a generating function of the Lagrangian variety $\La$. Namely, $\La$ lies the image of the map
\bean
\label{generate Psi}
(z_I, p_{\bar I}) \ \mapsto\ \big(z_I, z_{\bar I}=\frac{\der \Psi_I}{\der p_{\bar I}}(z_I, p_{\bar I}), p_I = - \frac{\der \Psi_I}{\der z_{I}}(z_I, p_{\bar I}), p_{\bar I}\big).
\eean

\end{thm}

\begin{proof}
The proof that these formulas give  \Ref{coord}  is by straightforward verification.
\end{proof}

\subsection {Integrals in involution}
Consider the standard Poisson bracket on $T^*(\C^n)$,
\bea
\{M,N\} = \sum_{j=1}^n \Big(\frac{\der M}{\der z_j}\frac{\der N}{\der p_j}       - \frac{\der M}{\der p_j}\frac{\der N}{\der z_j}\Big)
\eea
for  $M,N \in \mc O(T^*(\C^n))$. The functions are in involution if $\{M,N\}=0$.

\begin{thm}
\label{thm involution}
All functions $F_I(p)$, $|I|=k-1$, and $G_I(z,p)$, $|I|=k+1$, are in involution.
\end{thm}

\begin{proof} Clearly, $\{F_I,F_{I'}\}=0$, since $F_I,F_{I'}$ depend on $z$ only.
If $I=\{j_1,\dots,j_{k+1}\}$ and $I'=\{i_1,\dots, i_{k-1}\}$, then
\bea
\{G_I,F_{I'}\} = \sum_{m=1}^{k+1} (-1)^{m-1} d_{j_1,\dots,\widehat{j_m},\dots,j_{k+1}} d_{j_m,i_{1}\dots,i_{k-1}} =0
\eea
due to the Pl\"ucker relation \ref{Plu}.

Recall the function $G_j(z_j,p_j)$ in \Ref{G}. It is clear that $\{G_j,G_{j'}\}=0$ for all $j,j'\in J$. Now
$\{G_I,G_{I'}\}=0$ for all $I, I'$ with $|I|=|I'|=k+1$, since  $G_I, G_{I'}$ are linear combination of $G_j$ with constant coefficients.
\end{proof}

All the functions $F_I, G_I$ define commuting Hamiltonian flows, preserving $\La$ and giving symmetries of $\La$. For $I=\{i_1,\dots, i_{k-1}\}$,
the flow $\phi_I^t$ of the function $F_I(p)$ has the form
\bea
(z_1,\dots,z_n,p)\mapsto (z_1 + d_{1,i_1,\dots, i_{k-1}}t, \dots, z_n + d_{n,i_1,\dots, i_{k-1}}t, p).
\eea
For $I=\{j_1,\dots,j_{k+1}\}$, the flow $\phi_I^t$ of the function $G_I(z,p)$ does not change the pair of coordinate
$(z_j,p_j)$ of a point, if $j\notin I$, and transforms the pair $(z_{j_m},p_{j_m}) $ to the pair
\bea
 (z_{j_m} -\frac {a_{j_m}}{p_{j_m}} +\frac {a_{j_m}}{p_{j_m} + (-1)^m d_{j_1,\dots,\widehat{j_m},\dots,i_{k+1}}t} ,\
 p_{j_m} + (-1)^m d_{j_1,\dots,\widehat{j_m},\dots,i_{k+1}}t)\,
\eea
 for $m=1,\dots,k+1$.

\begin{rem}
An interesting property of the Hamiltonians $F_I, G_I$ is that they are regular with respect the Pl\"ucker coordinates
$d_{i_1,\dots,i_k}$.
Hence, they can be used to study the Langrange varieties of the arrangements in $\C^n\times\C^k$ associated with
not necessarily generic matrices $(b^i_j)$.
\end{rem}

\subsection{Hessian as a function on the Lagrange variety} Let $z\in\C^n-\Delta$ and $t^0$ a critical point
of the master function $\Phi(z,\,\cdot\,)$. An important characteristic of the critical point is the
Hessian
\bea
\on{Hess}\,\Phi(z,t^0) = \on{det}_{i,j=1}^k \Big(\frac{\der^2\Phi}{\der t_i\der t_j}(z,t^0 )   \Big) ,
\eea
see, for example, \cite{AGV, MV, V1, V2}.

For a subset $I=\{i_1,\dots,i_k\}\subset J$, we denote by $d_I^2$ the number  $(d_{i_1,\dots,i_k})^2$.

\begin{lem}
\label{lem Hess}
We have
\bean
\label{Hess}
\on{Hess}\,\Phi \,=\, (-1)^k
 \sum_{I\subset J, |I|=k} \! d^2_I \prod_{i\in I}^k \frac{p_{i}^2} {a_{i}} .
\eean
\end{lem}

\begin{proof}
In \cite{V2}, the formula $\on{Hess}\,\Phi = (-1)^k \sum_{1\leq i_1<\dots<i_k\leq n} d^2_{i_1,\dots,i_k}\prod_{m=1}^k a_{i_m}/ f^2_{i_m}$
is given, which is the right hand side of \Ref{Hess}. The formula itself is obvious.
\end{proof}

\subsection{Hessian and Jacobian}

Let $M=\{m_1,\dots,m_k\}\subset J$ be a $k$-element subset and $z_M,p_{\bar M}$ the corresponding ordered coordinate
system on $\La$. The functions $z_1,\dots,z_n$ form an ordered coordinate system on $\C^n-\Delta$.
Consider the projection $\La \mapsto \C^n-\Delta$, $(z,p)\mapsto z$, and the Jacobian
 $\on{Jac}_I(z_M,p_{\bar M})$ of the projection with respect to the chosen coordinate systems.



\begin{thm}
\label{thm hess jac}  As a function on $\La$, the function $d_M^2\on{Jac}_M$ does not depend on $M$ and
\bean
\label{jac hess}
d_M^2\on{Jac}_M =
 (-1)^{n-k}
 \sum_{L\subset J,\, |L|=n-k} d^2_{\bar L} \prod_{j\in L} \frac{a_j}{p_j^2} .
\eean

\end{thm}

\begin{proof}
The function $d_M^2\on{Jac}_M$ does not depend on $M$ by Lemma \ref {lem trans fns}.

Consider the function
$\tilde \Psi = \sum_{j\in J} a_j\ln p_j$
of $n$ variables $p_{j}$. Express in $\tilde \Psi$ the variables $p_M$ in terms of variables $p_{\bar M}$ by formulas  \Ref{coord}.
Denote by $\tilde \Psi_M(p_{\bar M})$ the resulting function.
By Theorem \ref{thm Gener},   $\on{Jac}_M = \on{det}\big(\frac {\der^2 \tilde \Psi_M}{\der p_{\bar M} \der p_{\bar M}}\big)$.
This implies that $d_M^2\on{Jac}_M$ is a polynomial in $a_j, j\in J$, of the form
\bea
d_M^2\on{Jac}_M = \sum_{L\subset J,\, |L|=n-k} c_L \prod_{j\in L} \frac{a_j}{p_j^2},
\eea
where $c_L$ are numbers independent of $M$. Our goal is to show that $c_L = (-1)^{n-k}d^2_{\bar L}$ but this is clear for $L=M$. This proves the theorem.
\end{proof}

\begin{cor}
\label{cor hess jac}
We have
\bean
\label{jac hess}
d_M^2\on{Jac}_M  = (-1)^n \Hess\, \Phi \,  \prod_{j\in J} \frac{a_j}{p_j^2}  \, .
\eean

\end{cor}

\section{Characteristic variety of the Gauss-Manin differential equations}
\label{sec last}

\subsection{Space $\Sing V$} Consider the complex vector space $V$ generated by vectors
$v_{i_1,\dots,i_k}$ with $i_1,\dots,i_k\in J$ subject to the relations $v_{i_{\si(1)},\dots, i_{\si(k)}}=
(-1)^\si v_{i_1,\dots,i_k}$ for any $i_1,\dots,i_k\in J$ and $\si\in S_k$. The vectors $v_{i_1,\dots,i_k}$
with $1\leq i_1<\dots<i_k\leq n$ form a basis of $V$.
If $v = \sum_{1\leq i_1<\dots<i_k\leq n} c_{i_1,\dots,i_k}v_{i_1,\dots,i_k}$ is a vector of $V$,
 we introduce the numbers $c_{i_1,\dots,i_k}$ for all $ i_1,\dots,i_k\in J$ by the rule:
 $c_{i_{\sigma(1)},\dots,i_{\sigma(k)}}=(-1)^\sigma c_{i_1,\dots,i_k}$.
 We introduce the subspace $\sing V\subset V$ of singular vectors by the formula
\bea
\sing V
 =
\Big\{ \sum_{1\leq i_1<\dots<i_k\leq n} c_{i_1,\dots,i_k}v_{i_1,\dots,i_k}\ | \ \sum_{j\in J} \,a_j\,c_{j,j_1,\dots,j_{k-1}}=0\
\text{for all}\ \{j_1,\dots,j_{k-1}\}\subset J \Big\}.
\eea
The symmetric bilinear contravariant form on $V$ is defined by the formulas:
\linebreak
$S(v_{i_1,\dots,i_k},v_{j_1,\dots,j_k}) = 0$,  if $ \{i_1,\dots,i_k\}\ne \{i_1,\dots,i_k\}$,
and
$S(v_{i_1,\dots,i_k},v_{i_1,\dots,i_k}) = \prod_{m=1}^k a_{i_m}$, if
$i_1,\dots,i_k$ are distinct.
Denote by  $s^\perp : V\to \sing V$ the orthogonal projection with respect to the contravariant form.

\subsection{Differential equations}

Consider the master function $\Phi(z,t)$ as a function  on $ U^0\subset \C^n\times \C^k$.
Let $\kappa$ be a nonzero complex number.
The function $e^{\Phi(z,t)/\kappa}$
defines a rank one local system $\mc L_\kappa$
on $U^0$ whose horizontal sections
over open subsets of $\tilde U$
 are univalued branches of $e^{\Phi(z,t)/\kappa}$ multiplied by complex numbers.
The vector bundle
\bea
\cup_{z\in \C^n-\Delta}\,H_k(U(\A(z)), \mc L_\kappa\vert_{U(\A(z))})\ \to\
 {} \ \C^n-\Delta
 \eea
has the canonical flat Gauss-Manin connection.  For a  horizontal section
\\
$\ga(z)\in$ $ H_k(U(\A(z)), \mc L_\kappa\vert_{U(\A(z))})$, consider the $V$-valued function
\bea
I_\ga(z) = \sum_{1\leq i_1<\dots<i_k\leq n} \big(\int_{\ga(z)}e^{\Phi(z,t)/\kappa} d \ln f_{i_1}\wedge \dots\wedge
 d \ln f_{i_k}\big) v_{i_1,\dots,i_k} .
 \eea
For any horizontal section $\ga(z)$, the function $I_\ga(z)$ takes values in $\sing V$ and satisfies the
Gauss-Manin differential equations
\bean
\label{GM}
\ka\frac{\der I_\ga}{\der z_j} = K_j(z) I_\ga, \qquad j\in J,
\eean
where $K_j(z) \in \End (\sing V)$ are suitable linear operators independent of $\ka$ and $\ga$.
Formulas for $K_j(z)$ see, for example,  in \cite[Formula (5.3)]{V3}.

For $z\in \C^n-\Delta$, the subalgebra $\B(z)\subset \End (\sing V)$ generated by the identity operator and the
operators $K_j(z), j\in J$, is called the Bethe algebra
at $z$ of the Gauss-Manin differential equations. The Bethe algebra is a maximal commutative subalgebra of $\End (\sing V)$,
see \cite[Section 8]{V3}.

We define the characteristic
variety of the $\kappa$-dependent $D$-module associated with the Gauss-Manin differential equations \Ref{GM} as
\be
\Spect = \{(z,p)\in T^*(\C^n-\Delta)\ |\ \exists
v\in \sing\,V\ \text{with}\ K_j(z)v = p_jv,\ j\in J\}.
\ee

\subsection{Identification}
Let $z\in\C^n-\Delta$.  By Lemma \ref{p properties}, given $j_1\in J$,
the  monomials $p_{i_1}\dots p_{i_k}$, with  $i_1<\dots<i_k$ and $j_1\notin\{i_1,\dots,i_k\}$,
form a $\C$-basis of $A_\Phi(z)$.
Consider the linear map $\mu : A_\Phi(z) \to \sing\,V$
which sends $d_{i_1,\dots,i_k}p_{i_1}\dots p_{i_k}$ to $s^\perp(v_{i_1,\dots, i_k})$ for all
 $i_1<\dots<i_k$ with $j_1\notin\{i_1,\dots,i_k\}$.

\begin{thm} [{\cite[Corollary 6.16]{V4}}]
\label{thm iso}
The linear map $\mu$ does not depend on $j_1$ and is an isomorphism  of complex vector spaces. For any $j\in J$, the isomorphism $\mu$
identifies the operator of multiplication by $p_j$ on $A_\Phi(z)$ and the operator $K_j(z)$ on $\sing\,V$.
\end{thm}

\begin{cor}
\label{cor Char=Lagr}
The characteristic variety $\Spect$ of the Gauss-Manin differential equations coincides with the Lagrangian variety of the master
function.
\end{cor}

Thus the statements in Section \ref{ sec Lagrangian var function} give us information on the characteristic variety of the Gauss-Manin differential equations.
In particular, equations in $A_\Phi(z)$ are satisfied in $\B(z)$, for example,
\bea
f_{i_1, i_2,\dots,i_{k+1}}(z)K_{i_1}(z)\dots K_{i_{k+1}}(z)
=
\sum_{m=1}^{k+1} (-1)^{m-1} a_{i_m} d_{i_1,\dots,\widehat{i_m},\dots,i_{k+1}} K_{i_1}(z)\dots \widehat{K_{i_m}(z)}\dots K_{i_{k+1}}(z).
\eea


\end{document}